\documentclass{article}

\title{The Chromatic Number of Finite Group Cayley Tables}
\author{Luis Goddyn\thanks{Department of Mathematics, Simon Fraser University, 8888 University Drive, Burnaby BC V5A 1S6, Canada.} \and Kevin Halasz\footnotemark[1] \and
 E.S. Mahmoodian\thanks{Department of Mathematical Sciences, Sharif University of Technology,
 P.O. Box 11155--9415, Tehran, I.R. Iran }}
\date{\today}

\usepackage[margin=1in]{geometry}
\usepackage{etoolbox}
\usepackage{fontenc}
\usepackage{geometry}
\usepackage{lmodern}
\usepackage{nowidow}
\usepackage{tocloft}
\usepackage{amsmath}                            
\usepackage{amssymb}                            
\usepackage{amsthm}                             
\usepackage{graphicx}                           
\usepackage[pdfborder={0 0 0}]{hyperref}                           
\usepackage{cases}   
\usepackage{multirow}
\usepackage{adjustbox}
\usepackage{xcolor}
\usepackage{tikz}
\usepackage{hyperref}
\usepackage{subcaption}

\newcommand{\modulo}[1]{\text{ (mod } #1\text{)}}
\newcommand{\defn}[1]{\textbf{#1}}

\DeclareMathOperator{\Supp}{Supp}

\providecommand{\keywords}[1]{\textbf{{Keywords}} #1}

\theoremstyle{plain}
\newtheorem{theorem}{Theorem}[section]
\newtheorem{conjecture}[theorem]{Conjecture}
\newtheorem{lemma}[theorem]{Lemma}
\newtheorem*{claim}{Claim}
\newtheorem{proposition}[theorem]{Proposition}

\begin{document}

\maketitle

\begin{abstract}
The chromatic number of a latin square $L$, denoted $\chi(L)$, is the minimum number of partial transversals needed to cover all of its cells. It has been conjectured that every latin square satisfies $\chi(L) \leq |L|+2$. If true, this would resolve a longstanding conjecture---commonly attributed to Brualdi---that every latin square has a partial transversal of size $|L|-1$. Restricting our attention to Cayley tables of finite groups, we prove two main results. First, we resolve the chromatic number question for Cayley tables of finite Abelian groups: the Cayley table of an Abelian group $G$ has chromatic number $|G|$ or $|G|+2$, with the latter case occurring if and only if $G$ has nontrivial cyclic Sylow 2-subgroups. Second, we give an upper bound for the chromatic number of Cayley tables of arbitrary finite groups. For $|G|\geq 3$, this improves the best-known general upper bound from $2|G|$ to $\frac{3}{2}|G|$, while yielding an even stronger result in infinitely many cases.
\end{abstract}


\section{Introduction and preliminaries}

Let $n$ be a positive integer, let $[n]:=\{0,1,2,\ldots,n-1\}$, and let $L$ be a \defn{latin square} of order $n$, which we define as an $n \times n$ array in which each row and each column is a permutation of some set of $n$ symbols indexed by $[n]$. Let $L$ be a latin square of order $n$. We define a \defn{partial transversal} of $L$ as a collection of cells which intersects each row, each column, and each symbol class at most once. A \defn{transversal} of $L$ is a partial transversal of size $n$, and a \defn{near transversal} is a partial transversal of size $n-1$. It is well known that $L$ possesses an orthogonal mate if and only if it can be partitioned into transversals. But when $L$ does not have an orthogonal mate, can we still efficiently partition its cells into partial transversals?

\begin{figure}[h]
\begin{center}
\begin{subfigure}{0.4\textwidth}{\LARGE$L=$}
\resizebox{4.5cm}{!}{
\begin{tabular}
{|c|c|c|}
  \hline
0&1&2\\ \hline
1&2&0\\ \hline
2&0&1\\ \hline
\end{tabular}
}
\end{subfigure}
\hspace{12pt}{\LARGE$\Gamma(L)=$}
\begin{subfigure}{0.3\textwidth}
\includegraphics[scale=0.75]{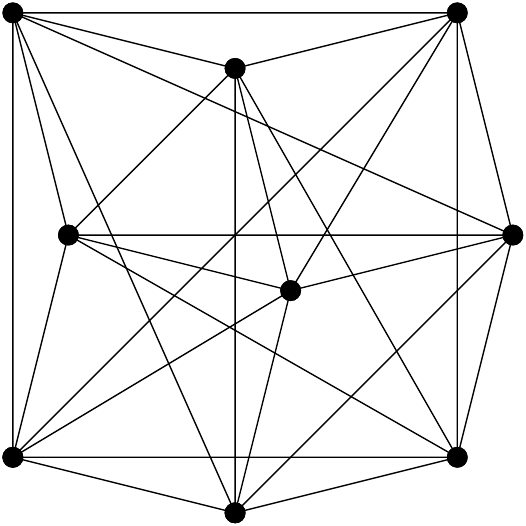}
\end{subfigure}
\caption{A latin square $L$ and its associated latin square graph $\Gamma(L)$.}\label{lsgz3}
\end{center}
\end{figure}

\vspace{-12pt}

This question can be restated in terms of graph coloring. Associated with every latin square $L$ is a strongly regular graph $\Gamma(L)$ defined on vertex set $\{(r,c,L_{r,c}) \,:\, r,c \in [n]\}$ with $(r_1,c_1,s_1) \sim (r_2,c_2,s_2)$ if and only if exactly one of $r_1=r_2$, $c_1=c_2$, or $s_1=s_2$ holds (see Figure \ref{lsgz3}). It is straightforward to check that partial transversals of $L$ correspond to independent sets in $\Gamma(L)$. Thus, the graph chromatic number $\chi(\Gamma(L))$ is the minimum number of partial transversals needed to cover all of the cells in $L$.

We refer to a partition of a latin square $L$ into $k$ partial transversals as a (proper) \defn{$k$-coloring} of $L$. The \defn{chromatic number} of $L$, denoted $\chi(L)$, is the minimum $k$ for which $L$ has a $k$-coloring. As a partial transversal has size at most $n$, $\chi(L) \geq n$. On the other hand, we may bound $\chi(L)$ from above by applying Brooks' theorem to the graph $\Gamma(L)$. 
 \begin{proposition}\label{fundthm}
 Let $L$ be a latin square of order $n\geq 3$. Then 
 \begin{equation*}n \leq \chi(L) \leq 3n-3,\label{fundineq}\end{equation*}
  with equality holding for the lower bound if and only if $L$ possesses an orthogonal mate. 
 \end{proposition}

Parts of this work have already appeared in the M.Sc.\! thesis of the second author \cite{HalaszMScThesis}, which was supervised by the first author. For any undefined terms we refer the reader to \cite{BondyMurty, DenesKeedwellBook, RobinsonBook}. Although latin square colorings are a natural generalization of the notion of possessing an orthogonal mate, the concept did not appear in the literature until very recently. In 2015, papers introducing the concept were submitted by two separate groups: Cavenagh and Kuhl \cite{CavenaghKuhl2015} and Besharati et al.\! \cite{Luisetal2016} (a group containing the first and third author of the present paper). The two groups independently conjectured the following upper bound.

\begin{conjecture}\label{chromaticconjecture}
Let $L$ be a latin square of order $n$. Then
\begin{equation*}\label{conjchrombnd}\chi(L) \leq \begin{cases} n+1 &\mbox{if } n \text{ is odd,}\\ n+2 &\mbox{if } n \text{ is even.} \end{cases} \end{equation*}
\end{conjecture}

Latin squares for which this conjecture is tight have been given by Euler \cite{Euler} in the even case and by Wanless and Webb \cite{WanlessWebb2006} in the odd case. If true, Conjecture \ref{chromaticconjecture} is likely difficult to prove, as it implies a pair of long-standing conjectures concerning the existence of large partial transversals in latin squares. These conjectures are attributed to Brualdi and Ryser, respectively. 

\begin{conjecture}[\cite{DenesKeedwellBook, Ryser1967,Stein1975}]\label{RBSconj}
Let $L$ be a latin square of order $n$. Then (1) $L$ possesses a near transversal, and (2) if $n$ is odd then $L$ possesses a transversal.
\end{conjecture}

To see that Conjecture \ref{chromaticconjecture} implies Conjecture \ref{RBSconj}.1, suppose there exists a latin square $L$ in which every partial transversal has size at most $n-2$. Any set of $n+2$ partial transversals in $L$ could cover at most $(n+2)(n-2) = n^2-4$ of $L$'s $n^2$ cells. Thus, $L$ has no proper $(n+2)$-coloring. A similar argument shows that Conjecture \ref{chromaticconjecture} implies Conjecture \ref{RBSconj}.2.

There is a growing body of evidence for Conjecture \ref{chromaticconjecture}. As noted in \cite{Luisetal2016}, the conjecture has been verified for $n \leq 8$. It is also known to hold in an asymptotic sense. Using the canonical representation of latin squares as 3-uniform hypergraphs, the following is an immediate corollary of a powerful theorem due to Pippenger and Spencer \cite{PipSpencer1989}.
\begin{theorem}\label{pipspenc}
As $n \rightarrow \infty$, every latin square $L$ of order $n$ satisfies $\chi(L) = n+o(n)$.
\end{theorem}
Furthermore, there are several families for which Conjecture \ref{chromaticconjecture} is known to hold. Cavenagh and Kuhl \cite{CavenaghKuhl2015} showed that Conjecture \ref{chromaticconjecture} holds for circulant latin squares (i.e.\! Cayley tables of cyclic groups) when $n \not\equiv 6 \modulo{12}$. This result was confirmed and extended in \cite{Luisetal2016}, where Conjecture \ref{chromaticconjecture} was established for circulant latin squares of every order.

The present paper continues the work towards resolving Conjecture \ref{chromaticconjecture} in the special case where $L$ is the Cayley table of a finite group. Observe that the chromatic number of $\Gamma := \Gamma(L)$ is not affected by relabelling the rows, columns, or symbol classes of $L$, nor is it affected by applying a fixed permutation to each of the triples $(r,c,s) \in V(\Gamma)$. Thus, $\chi(L)$  is a \emph{main class invariant}, and it makes sense in this context to speak of \emph{the} Cayley table of a group $G$, which we denote by $L(G)$. In a slight abuse of notation, we write $\chi(G)$ for the chromatic number of $L(G)$ and $\Gamma(G)$ for the latin square graph $\Gamma(L(G))$. 

Given a group $G$, let $Syl_2(G)$ denote the isomorphism class of its Sylow 2-subgroups. The groups for which $\chi(G)=n$ were recently characterized by Bray, Evans, and Wilcox \cite{Evans2009,Wilcox2009}, resolving a 50 year old conjecture due to Hall and Paige \cite{HallPaige1955}. 
\begin{theorem}\label{bigequiv}
Let $G$ be a group of order $n$. Then the following are equivalent:
\begin{enumerate}
\item $\chi(G) = n$,
\item $\chi(G) \leq n+1$,
\item $L(G)$ has a transversal,
\item $Syl_2(G)$ is either trivial or non-cyclic.
\end{enumerate}
\end{theorem}
In light of this, verifying Conjecture \ref{chromaticconjecture} amounts to showing that every group with nontrivial, cyclic Sylow 2-subgroups has an $(n+2)$-coloring. In Section \ref{abelian} we give such a construction under the additional assumption that $G$ is Abelian, yielding our first main result.
\begin{theorem}\label{abelianchrom}
Let $G$ be an Abelian group of order $n$. Then
\begin{equation} \chi(G) = \begin{cases} n &\mbox{if } Syl_2(G) \text{ is either trivial or non-cyclic,}\\ n+2 &\mbox{otherwise}. \end{cases} \label{mainthm1}\end{equation}
\end{theorem}

In Section \ref{generalbd} we turn to the case of general (i.e.\! not necessarily Abelian) groups. We begin by showing that $\chi(G)$ is submultiplicative, thereby generalizing a classical result due to Hall and Paige.
This allows us to establish an upper bound for $\chi(G)$ which depends only upon the largest power of 2 dividing $|G|$: every group $G$ of order $n = 2^lm \geq 3$ satisfies 
\begin{equation}
\chi(G) \leq n+\frac{n}{2^{l-1}}.
\label{twoss}\end{equation}
 Previously, the best known general upper bound for $\chi(G)$ was due to Wanless; it follows directly from his work in \cite{Wanless2002} that every finite group satisfies $\chi(G) \leq 2n$. This bound is improved upon by \eqref{twoss} except when $n \equiv 2 \modulo{4}$. Dealing directly with this final case, we obtain our second main result.
\begin{theorem}\label{brooksbound}
Let $G$ be a group of order $n \geq 3$. Then
\[ \chi(G) \leq \frac{3}{2}n.\]
\end{theorem}
The condition $n \geq 3$ is necessary because $\Gamma(\mathbb{Z}_2) \cong K_4$. It is worth noting that
Theorem \ref{brooksbound} is still somewhat far from the best possible. Indeed, if Conjecture \ref{chromaticconjecture} is true, then the bound \eqref{mainthm1} holds for every finite group.

\section{The chromatic number of Abelian groups\label{abelian}}

Let $G=\mathbb{Z}_{n_1} \times \mathbb{Z}_{n_2} \times\cdots  \times  \mathbb{Z}_{n_k}$ be a finite Abelian group of order $n$. We say that $G = \{g_0,g_1,\ldots,g_{n-1}\}$ is ordered \defn{lexicographically} if $g_i = (i_1,i_2,\ldots,i_k)$ precedes $g_j = (j_1,j_2,\ldots,j_k)$ (i.e.\! $i<j$) if and only if there is some $l\in \{1,2,\ldots,k\}$ for which $i_l < j_l$ and $i_m = j_m$ for every positive integer $m < l$. In the statement of the following technical lemma, indices are expressed modulo $n$.

\begin{lemma}
Let $G=\mathbb{Z}_{n_1} \times \mathbb{Z}_{n_2} \times\cdots  \times  \mathbb{Z}_{n_k}=\{g_0,g_1,\ldots,g_{n-1}\}$ be a lexicographically ordered Abelian group of odd order $n$. If $\gcd(s+1,n) = 1$ for some positive integer $s$, then the map $\phi: G \rightarrow G$ given by $\phi(g_i) = g_{i+c}+g_{si+d}$ is injective for every $c,d\in[n]$.
\label{techlemma}\end{lemma}

\begin{proof}

Suppose $G = \mathbb{Z}_n$ is cyclic and consider $g,h \in G$ such that $\phi(g) = \phi(h)$. We may treat $g$ and $h$ as integers in the set $[n]$, in which case the group operation is simply addition modulo $n$, and
 \[ (s+1)g+c+d \equiv \phi(g) \equiv \phi(h) \equiv (s+1)h +c+d\modulo{n}.\]
We then have $(s+1)g \equiv (s+1)h \modulo{n}$. But $\gcd(s+1,n)=1$ tells us that $s+1$ is a generator of $G = \mathbb{Z}_n$, and therefore $g \equiv h \modulo{n}$, as desired. 

Now, we may assume $G = \mathbb{Z}_{n_1} \times H$, where $H=\mathbb{Z}_{n_2} \times\cdots  \times  \mathbb{Z}_{n_k}$ is a nontrivial Abelian group of odd order $m := n/n_1$. If $H = \{h_0,h_1,\ldots,h_{m-1}\}$ is ordered lexicographically, then for every $i \in [n]$
\begin{equation*}
g_i = \left(\left\lfloor\frac{i}{m}\right\rfloor, h_{i \modulo{m}} \right)
\label{rootbijorder}\end{equation*}
and, defining the map $\psi: H \rightarrow H$ by $\psi(h_i) = h_{i+c\modulo{m}}+ h_{si+d\modulo{m}}$, we have
\begin{equation*} \phi(g_i) = \left(\left\lfloor \frac{i+c}{m} \right\rfloor+\left\lfloor \frac{si+d}{m} \right\rfloor \modulo{n_1}\, ,\; \psi\left(h_{i \modulo{m}}\right)\right).\label{inductphi}\end{equation*}
Consider $i,j \in [n]$ such that $\phi(g_i) = \phi(g_j)$. In this case we have $\psi\left(h_{i \modulo{m}}\right) = \psi\left(h_{j \modulo{m}}\right)$. It then follows by induction on $|G|$ that $i \equiv j \modulo{m}$. Thus, there is some $r \in [n_1]$ such that $j = i+rm$, and 
\begin{equation*}
 \left\lfloor \frac{i+c}{m} \right\rfloor+\left\lfloor \frac{si+d}{m} \right\rfloor \equiv \left\lfloor \frac{i+rm+c}{m} \right\rfloor+\left\lfloor \frac{si+srm+d}{m} \right\rfloor\equiv \left\lfloor \frac{i+c}{q} \right\rfloor+r+\left\lfloor \frac{si+d}{q} \right\rfloor +sr\modulo{n_1}.
\end{equation*}
We then have $(s+1)r \equiv 0 \modulo{n_1}$. But $r \in [n_1]$ and $\gcd(s+1,n_1) = 1$, forcing us to conclude that $r=0$, in which case $i=j$. 
\end{proof}

The \defn{M\"{o}bius ladder} of order $2n$, denoted $M_n$, is the cubic graph formed from a cycle of length $2n$ by adding $n$ edges, one between each pair of vertices at distance $n$ in the initial cycle. We refer to this initial cycle as the \emph{rim} of $M_n$, and refer to the edges between opposite vertices in the rim as \emph{rungs}. A pair $\{u,v\} \subseteq V(M_n)$ is called \defn{near-antipodal} if the shortest path from $u$ to $v$ along the rim of $M_n$ has length $n-1$ (see Figure \ref{mobladz}). There is a strong sense in which M\"{o}bius ladders are ``nearly" bipartite.

\begin{proposition}
For $n \geq 3$, let $M_n=(V,E)$ be the M\"{o}bius ladder of order $2n$, and let $\{u,v\} \in V$ be a near-antipodal pair. Then the induced subgraph $M_n\left[V\setminus\{u,v\}\right]$ is bipartite.
\label{greedymob}\end{proposition}
\begin{proof}
Observe that the greedily coloring of $M_n\left[V\setminus\{u,v\}\right]$ with vertices ordered clockwise around the rim of $M_n$ uses exactly 2 colors. 
\end{proof}

\begin{figure}[h]
\begin{center}
\begin{tikzpicture}[scale=0.7, every node/.style={scale=0.7}]
\def \n {36}
\def \radius {3cm}
\def \margin {0} 

\foreach \s in {1,...,\n}
{
  \node[fill=black, circle,inner sep=0pt,minimum size=2mm] at ({360/\n * (\s - 1)}:\radius) (\s){};
  \draw ({360/\n * (\s - 1)+\margin}:\radius) 
    arc ({360/\n * (\s - 1)+\margin}:{360/\n * (\s)-\margin}:\radius);
}

\draw (1) -- (19);
\draw (2) -- (20);
\draw (3) -- (21);
\draw (4) -- (22);
\draw (5) -- (23);
\draw (6) -- (24);
\draw (7) -- (25);
\draw (8) -- (26);
\draw (9) -- (27);
\draw (10)--(28);
\draw (11)--(29);
\draw (12)--(30);
\draw (13)--(31);
\draw (14)--(32);
\draw (15)--(33);
\draw (16)--(34);
\draw (17)--(35);
\draw (18)--(36);

=

  \node[fill=red, circle,inner sep=0pt,minimum size=4mm] at ({360/\n * (1 - 1)}:\radius) {$v$};
  \node[fill=red, circle,inner sep=0pt,minimum size=4mm] at ({360/\n * (20 - 1)}:\radius) {$u$};

\end{tikzpicture}
\hspace{36pt}
\begin{tikzpicture}[scale=0.7, every node/.style={scale=0.7}]
\def \n {18}
\def \radius {2.2cm}
\def \margin {0} 

\foreach \s in {1,...,\n}
{
  \node[fill=black, circle,inner sep=0pt,minimum size=2mm] at ({360/\n * (\s - 1)}:\radius) (a\s){};
 }

\def \radius {3cm}
\foreach \s in {1,...,\n}
{
  \node[fill=black, circle,inner sep=0pt,minimum size=2mm] at ({360/\n * (\s - 1)}:\radius) (b\s){};
}

\foreach \s in {1,...,17}  
{
\foreach \s [evaluate=\s as \r using int(\s+1)] in {1,...,17}
{
\draw (b\s)--(b\r);
\draw (a\s)--(a\r);
\draw (a\s)--(b\s);
\draw (a\r)--(b\r);
  }
}

\draw (a18)--(b1);
\draw (b18)--(a1);

  \node[fill=red, circle,inner sep=0pt,minimum size=4mm] at ({360/\n * (5 - 1)}:\radius) {$v$};
    \node[fill=red, circle,inner sep=0pt,minimum size=4mm] at ({360/\n * (6 - 1)}:2.2cm) {$u$};

\end{tikzpicture}
\caption{Two drawings of the M\"{o}bius ladder $M_{18}$ with a near-antipodal pair of vertices highlighted.}\label{mobladz}
\end{center}
\end{figure}
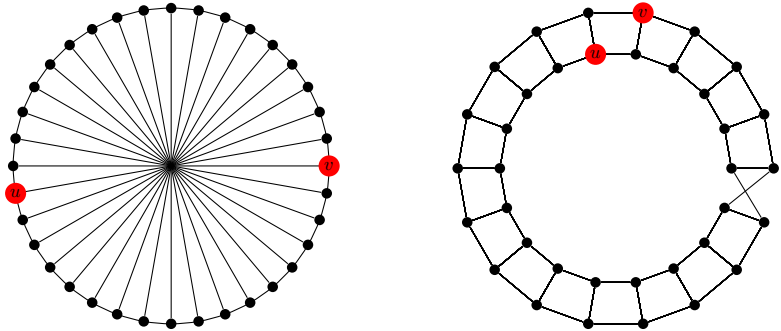

\vspace{-12pt}

Let $G=\{g_0,g_1,\ldots,g_{n-1}\}$ be a group satisfying $\chi(G) > n$. We construct an $(n+2)$-coloring of $\Gamma = \Gamma(G)$ in two steps. First, we show that $\Gamma$ can be partitioned into $\frac{n}{2}$ induced copies of $M_n$. Second, we find a bipartite induced subgraph $\Lambda \subseteq \Gamma$ which contains a near-antipodal pair from each copy of $M_n$. By Proposition \ref{greedymob}, we then have a partition of $\Gamma$ into $\frac{n}{2}+1$ bipartite induced subgraphs. Using a disjoint pair of colors for each of these subgraphs, we obtain an $(n+2)$-coloring of $\Gamma$. 

Before formally presenting our construction, we introduce some notation which will be utilized both here and in Section \ref{generalbd}. Letting $L = L(G)$ and fixing an integer $d \in [n]$, we define the \defn{\emph{d}th right diagonal of $L$} as the set
\begin{equation*}\label{dthdiag}  T^L_d := \{ L_{i,i+d} \,:\, i \in [n]\},\end{equation*}
where indices are expressed modulo $n$. When it is clear which latin square we are discussing, we drop the superscript and simply write $T_d$. We also define the maps $R,C: L \rightarrow [n]$ and $S: L \rightarrow G$ by 
\begin{equation}R(L_{ij}) = i, C(L_{ij})=j,\text{ and }S(L_{ij}) = g_ig_j.\label{projmaps}\end{equation}
These maps send a cell of $L$ to its row index, its column index, and its symbol, respectively. We then extend these functions to sets of cells. For $A \subseteq L$, let $R(A) = \{R(a) \,:\, a \in A\}$ be the multiset containing the row-index of every cell in $A$ (counted with multiplicity), and define $C(A)$ and $S(A)$ similarly.

\begin{theorem}
\label{abelianmainresult}
Let $G$ be an Abelian group of order $n$. If $Syl_2(G)$ is cyclic and nontrivial then
\[\chi(G) \leq n+2.  \]
\end{theorem}

\begin{proof}

Because $Syl_2(G)$ is nontrivial, $n$ is even and the constant
\[ q:=n/2\]
is well-defined. We may assume $n \geq 4$, as $\Gamma(\mathbb{Z}_2) \cong K_4$ has chromatic number $4 = 2+2$. Moreover, letting $t:=|Syl_2(G)|$, there is some integer $l \geq 1$ such that $t = 2^l$. By the fundamental theorem of finite Abelian groups $G = \mathbb{Z}_t \times \mathbb{Z}_{m_1} \times \mathbb{Z}_{m_2} \times\cdots  \times  \mathbb{Z}_{m_k}$, where $m:= \prod_{i=1}^k m_k$ is odd. Letting $H :=   \mathbb{Z}_{m_1} \times \mathbb{Z}_{m_2} \times\cdots  \times  \mathbb{Z}_{m_k}$, we order $H = \{h_0,h_1,\ldots,h_{m-1}\}$ lexicographically. We then impose an ordering on $G$ by setting 
\begin{equation} g_i := \left( i \modulo{t}, h_{i \modulo{m}}\right) \text{ for every }i \in [n].\label{mainresorder} \end{equation}
 Arrange the rows and columns of $L = L(G)$ according to this ordering and, for every $i \in [q]$, define the set
 \[ D_i := T_{2i} \cup T_{2i+1}.\] 
 
Letting $\Gamma := \Gamma(G)$, recall that $V = V(\Gamma)$ corresponds to the set of cells in $L$, and that two cells are adjacent if they lie in the same row, they lie in the same column, or they contain the same symbol. By definition, two cells in a latin square can satisfy at most one of these conditions. We therefore have a natural partition of $E = E(\Gamma)$ into the sets $E_R$, $E_C$, and $E_S$, corresponding to ``row-edges,'' ``column-edges,'' and ``symbol-edges,'' respectively. 

\begin{claim}
The induced subgraph $\Gamma_i := \Gamma[D_i]$ is isomorphic to the M\"{o}bius ladder $M_n$ for every $i \in [q]$. 
\end{claim}
Expressing indices modulo $n$, let $A_j := L_{j,j+2i}$ and $B_j := L_{j,j+2i+1}$ for every $j \in [n]$. It is not hard to see that the only row-edges in $\Gamma_i$ are $\{A_jB_j \,:\, j \in [n]\}$ (see Figure \ref{xandy}). Similarly, the only column-edges in $\Gamma_i$ are $\{A_{j+1}B_j\,:\, j \in [n]\}$. Thus, the vertex sequence
\begin{equation}\label{ham}A_0,B_0,A_1,B_1,\ldots,A_{n-1},B_{n-1}\end{equation}
corresponds to a Hamilton cycle in $\Gamma_i$ that uses all of the edges in $E(\Gamma_i) \cap (E_R \cup E_C)$. To prove the claim, it is left to show that $E(\Gamma_i) \cap E_S$ contains exactly $n$ edges, each of which connects opposite vertices (i.e.\! vertices at distance $q$) in the cycle given by \eqref{ham}.

Given an integer $z$, let $\overline{z}$ be the corresponding residue modulo $t$. It follows from \eqref{mainresorder} that, for every $j \in [n]$, 
\begin{equation} S(A_j) = \left(\overline{2j+2i},h_j+h_{j+2i \modulo{m}}\right) \text{ and } S(B_j) = \left(\overline{2j+2i+1}, h_j+h_{j+2i+1\modulo{m}}\right). \label{aidef}\end{equation}
Because $t$ is even, for every $j,k \in [n]$ the first coordinates of $S(A_j)$ and $S(B_k)$ have different values modulo 2. Thus, $E_S$ contains no edges of the form $A_jB_k$. 

Fix $j \in [n]$ and suppose there is some nonzero $x \in [n]$ such that $S(A_j)=S(A_{j+x})$. It follows from \eqref{aidef} that $2j + 2i \equiv 2j + 2x + 2i \modulo{t}$. Recalling that $t = 2^l$, we may conclude that $x$ is divisible by $2^{l-1}$. On the other hand, \eqref{aidef} also implies $h_j+h_{j+2i} = h_{j+x }+h_{j+x+2i}$ (where indices are here expressed modulo $m$). Applying Lemma \ref{techlemma} with $c=0$, $d=2i$, and $s=1$, we have $j \equiv j+x \modulo{m}$, so that $x$ is divisible by $m$. Because $m$ is odd, the only nonzero integer in $[n]$ which is divisible by both $m$ and $2^{l-1}$ is $q = m2^{l-1}$. 

Observing that $S(A_{j}) = S(A_{j+q})$ for every $j \in [n]$, we see that each $A_j$ is incident to exactly one edge in $ E_S \cap E(\Gamma_i)$: the edge connecting it to the opposite vertex in the cycle given by \eqref{ham}. A similar argument shows that, for every $j \in [n]$, the only edge in $ E_S \cap E(\Gamma_i)$ incident to $B_j$ is $B_jB_{j+q}$. This establishes the claim.

We complete the proof (of Theorem \ref{abelianmainresult}) by finding a pair of independent sets $X,Y \subseteq V$ such that $\Gamma_i^\prime = \Gamma[D_i\setminus(X \cup Y)]$ is bipartite for every $i \in [q]$. Given such $X$ and $Y$, we properly $(n+2)$-color $\Gamma$ using a distinct pair of colors for each of the $\frac{n}{2}+1$ sets $D_0^\prime,D_1^\prime,\ldots,D_{q-1}^\prime,X \cup Y$. 

Towards a definition of $X$ and $Y$, let
\begin{equation} k:= \left\lceil \frac{n}{4} \right\rceil \text{ and } (q_0,q_1) := \begin{cases} (q,q+1)  &\mbox{if } q\equiv 0 \modulo{3}, \\ (q-1,q) &\mbox{otherwise}. \end{cases}\label{kqq}\end{equation}
Keeping in mind that indices are here considered modulo $n$, for each $i \in [k]$ let
\begin{equation}  x_i := L_{i,\,3i},\, x_i^\prime := L_{q_0+i,\,q_1+3i}, \text{ and } X := \{x_i,x_i^\prime \,:\, i \in [k]\}. \label{defnx}\end{equation}
Similarly, for every $j \in [q-k]$, define
\begin{equation}  y_j := L_{j,\,3j+2k},\, y_j^\prime := L_{q_0+j,\,q_1+3j+2k}, \text{ and } Y := \{y_j,y_j^\prime \,:\, j \in [q-k]\}. \label{defny}\end{equation}
Figure \ref{xandy} exhibits $X$ and $Y$ for the group $\mathbb{Z}_2 \times \mathbb{Z}_3 \times \mathbb{Z}_3$. Observe that $x_i \in T_{2i} \subseteq D_i$ and $x_i^\prime \in T_{2i+1} \subseteq D_i$ for each $i \in [k]$. Similarly, $y_j,y_j^\prime  \in  D_{k+j}$  for each $j \in [q-k]$. 

Recall that the edges on the rim of $\Gamma_i$ are exactly $E(\Gamma_i) \cap \left( E_R \cup E_C\right)$. It follows from the definition of $D_i$ that the shortest path from $x_i$ to $x_i^\prime$ along the rim of $\Gamma_i$ has length $n-1$. Similarly, the shortest path from $y_j$ to $y_j^\prime$ along the rim of $\Gamma_{k+j}$ has length $n-1$. Thus, $x_i,x_i^\prime$ and $y_j,y_j^\prime$ are near-antipodal pairs for every $i,j \in [k]$. Proposition \ref{greedymob} then implies that $\Gamma_i^\prime$ is bipartite for every $i \in [q]$. 

\begin{figure}
\begin{center}
\newcommand{\overred}[1]{{{\scriptsize{\textcolor{red}{\emph{$\widetilde{#1}$}}}}}}
\newcommand{\underblue}[1]{{\footnotesize{{\textcolor{blue}{{\emph{#1}}}}}}}
\newcommand{\blank}[1]{\textcolor{white}{#1}}
\def\arraystretch{1.05}
\adjustbox{scale=.8}{
\begin{tabular}
{|@{\hspace{2pt}}c@{\hspace{2pt}}|@{\hspace{2pt}}c@{\hspace{2pt}}|@{\hspace{2pt}}c@{\hspace{2pt}}|@{\hspace{2pt}}c@{\hspace{2pt}}|@{\hspace{2pt}}c@{\hspace{2pt}}|@{\hspace{2pt}}c@{\hspace{2pt}}|@{\hspace{2pt}}c@{\hspace{2pt}}|@{\hspace{2pt}}c@{\hspace{2pt}}|@{\hspace{2pt}}c@{\hspace{2pt}}|@{\hspace{2pt}}c@{\hspace{2pt}}|@{\hspace{2pt}}c@{\hspace{2pt}}|@{\hspace{2pt}}c@{\hspace{2pt}}|@{\hspace{2pt}}c@{\hspace{2pt}}|@{\hspace{2pt}}c@{\hspace{2pt}}|@{\hspace{2pt}}c@{\hspace{2pt}}|@{\hspace{2pt}}c@{\hspace{2pt}}|@{\hspace{2pt}}c@{\hspace{2pt}}|@{\hspace{2pt}}c@{\hspace{2pt}}|}
  \hline
\overred{\mathbf{000}}&\textbf{101}&002&110&011&112&020&121&022&100&\underblue{001}&102&010&111&012&120&021&122\\ \hline
101&\textbf{002}&\textbf{100}&\overred{011}&112&010&121&022&120&001&102&000&111&\underblue{012}&110&021&122&020\\ \hline
002&100&\textbf{001}&\textbf{112}&010&111&\overred{022}&120&021&102&000&101&012&110&011&122&\underblue{020}&121\\ \hline
110&\underblue{011}&112&\textbf{020}&\textbf{121}&022&100&001&102&\overred{010}&111&012&120&021&122&000&101&002\\ \hline
011&112&010&121&\textbf{022}&\textbf{120}&001&102&000&111&012&110&\overred{021}&122&020&101&002&100\\ \hline
112&010&111&022&120&\textbf{021}&\textbf{102}&000&101&012&110&011&122&020&121&002&100&001\\ \hline
020&121&022&100&001&102&\textbf{010}&\textbf{111}&012&120&021&122&000&101&002&110&011&112\\ \hline
121&022&120&001&102&000&111&\textbf{012}&\textbf{110}&021&122&020&101&002&100&011&112&010\\ \hline
022&120&021&102&000&101&012&110&\textbf{011}&\textbf{122}&020&121&002&100&001&112&010&111\\ \hline
100&001&\underblue{102}&010&111&012&120&021&122&\textbf{000}&\overred{\mathbf{101}}&002&110&011&112&020&121&022\\ \hline
001&102&000&111&012&\underblue{110}&021&122&020&101&\textbf{002}&\textbf{100}&011&\overred{112}&010&121&022&120\\ \hline
102&000&101&012&110&011&122&020&\underblue{121}&002&100&\textbf{001}&\textbf{112}&010&111&022&\overred{120}&021\\ \hline
010&\overred{111}&012&120&021&122&000&101&002&110&011&\underblue{112}&\textbf{020}&\textbf{121}&022&100&001&102\\ \hline
111&012&110&021&\overred{122}&020&101&002&100&011&112&010&121&\textbf{022}&\textbf{120}&001&102&000\\ \hline
012&110&011&122&020&121&002&100&001&112&010&111&022&120&\textbf{021}&\textbf{102}&000&101\\ \hline
120&021&122&000&101&002&110&011&112&020&121&022&100&001&102&\textbf{010}&\textbf{111}&012\\ \hline
021&122&020&101&002&100&011&112&010&121&022&120&001&102&000&111&\textbf{012}&\textbf{110}\\ \hline
\textbf{122}&020&121&002&100&001&112&010&111&022&120&021&102&000&101&012&110&\textbf{011}\\ \hline
\end{tabular}
}
\caption{A Cayley table of $\mathbb{Z}_2 \times \mathbb{Z}_3 \times \mathbb{Z}_3$ with elements of \overred{X}, \underblue{Y}, and $\mathbf{D_0}$ highlighted.}\label{xandy}
\end{center}
\end{figure}

It remains to show that $X$ and $Y$ are independent sets in $\Gamma$. We begin by showing that there are no row-edges and no column-edges between cells in $X$. Recalling the definitions in \eqref{projmaps} and \eqref{defnx}, we see that the multiset of row-indices of cells in $X$ is
\[R(X) = [k] \cup \{q_0+i \,:\, i \in [k]\}.\]
But, having assumed $n \geq 4$, we have $k-1 < q_0$ and $q_0+k-1 <n$. It follows that $R(X)$ is \emph{simple}--it contains no repeated entries. Now, define 
\begin{equation*}\widehat{X} := \{x_i \,:\, i \in [k]\}\text{ and }X^\prime := \{x_i^\prime \,:\, i \in [k]\}.\label{sepx}\end{equation*}
Again looking to \eqref{defnx}, we see that
 \[ C(\widehat{X}) = \{3i \,:\, i \in [k]\}  \text{ and } C(X^\prime)= \{q_1+3i \modulo{n} \,:\, i \in [k]\}.\]
 Because $3(k-1) < n$, both $C(\widehat{X})$ and $C(X^\prime)$ are simple sets. Thus, $C(X) = C(\widehat{X}) \cup C(X^\prime)$ is simple unless $C(\widehat{X}) \cap C(X^\prime) \neq \emptyset$. 
 Suppose there were some $x \in C(\widehat{X}) \cap C(X^\prime)$. As $x \in C(\widehat{X})$, there is an $i_0 \in [k]$ such that $x = 3i_0$, which implies $x \equiv 0 \modulo{3}$. On the other hand, $x \in C(X^\prime)$ means $x = q_1+3i_1 \modulo{n}$ for some $i_1 \in [k]$.  We claim that this implies $x \not\equiv 0 \modulo{3}$, yielding a contradiction. 
 
To prove this claim, first suppose that $n$ is a multiple of 3. In this case $q$ is also divisible by 3, and \eqref{kqq} tells us that $q_1 = q+1 \equiv 1 \modulo{3}$. But then $x \equiv q_1 + 3i_1 \equiv 1 \modulo{3}$. So, we may assume $n$ is not divisible by 3. In this case  \eqref{kqq} tells us that $q_1 = q \not\equiv 0 \modulo{3}$. When $q +3i_1 < n$ this implies $x = q+3i_1 \not\equiv 0 \modulo{3}$, while when $q+3i_1 \geq n$ we have
 \[x = q + 3i_1 - n \equiv q-n = -q \not\equiv 0 \modulo{3}.\]

Now, observe that $X$ and $Y$ have the same ``shape'' in $L$ in the sense that $R(Y) \subseteq R(X)$ and $C(Y) \subseteq \{c+2k \,:\, c \in C(X)\}$. Thus, having shown that $R(X)$ and $C(X)$ are simple, we may conclude that $R(Y)$ and $C(Y)$ are also simple. In other words, there are no row-edges or column-edges between cells in $Y$. 
 
We next show that $S(X)$ is a simple set. From here to the end of the proof indices are expressed modulo $m$. Recalling \eqref{kqq}, observe that $q_0+q_1 \in \{n-1,n+1\}$. Combinined with \eqref{mainresorder}, \eqref{defnx}, and the fact that $t$ divides $n$, this implies the existence of some $w \in \{-1,1\}$ such that, for every $i \in [k]$, 
\begin{equation*} S(x_i) = (\overline{4i} ,h_{i}+h_{3i}) \text{ and } S(x_i^\prime) = (\overline{4i+w} , h_{i+q_0} + h_{3i+q_1}).\label{sxdef}\end{equation*}
Considering the parity of entires in the first coordinate, we immediately see that $S(x_i) \neq S(x_j^\prime)$ for every $i,j \in [k]$. Thus, $S(\widehat{X}) \cap S(X^\prime) = \emptyset$. 

To see that $S(\widehat{X})$ is simple, consider $x_i,x_j \in \widehat{X}$ such that $S(x_i) = S(x_j)$. We then have $h_i+h_{3i} = h_j + h_{3j}$, and applying Lemma \ref{techlemma} with $c=d=0$ and $s=3$ tells us that $i \equiv j \modulo{m}$. We also have 
\begin{equation}
4i \equiv 4j \modulo{t}.
\label{fours}\end{equation}
Suppose $t \leq 4$. In this case \eqref{fours} is trivially satisfied. However, because $0 \leq i,j \leq k-1$,
\[ |i-j| < k =  \left\lceil \frac{mt}{4} \right\rceil \leq \left\lceil \frac{m4}{4} \right\rceil = m.\]
As distinct numbers are congruent modulo $m$ only if their difference is at least $m$, we may conclude that $i=j$. So, recalling that $t = 2^l$ for some integer $l \geq 1$, we may assume $t \geq 8$. It then follows from \eqref{fours} that $i -j \equiv 0 \modulo{2^{l-2}}$. Because $m$ is odd, $\gcd\left(m,2^{l-2}\right)=1$ and the Chinese Remainder Theorem tells us that $x=0$ is the unique $x \in [2^{l-2}m]$ satisfying $x \equiv 0 \modulo{2^{l-2}}$ and $x \equiv 0 \modulo{m}$. But we have just shown that $|i-j| \equiv 0 \modulo{m}$ and $|i-j| \equiv 0 \modulo{2^{l-1}}$. Thus, as $0 \leq |i-j| < k = 2^{l-2}m$, we have $i=j$.

A similar argument shows that $S(X^\prime)$ is simple. Indeed, when $S(x_i^\prime) = S(x_j^\prime)$, applying Lemma \ref{techlemma} with $c = q_0$, $d=q_1$, and $s = 3$ yields $i \equiv j \modulo{m}$, while $4i + w \equiv 4j + w \modulo{t}$ implies $4i \equiv 4j \modulo{t}$. From here we may proceed exactly as above. 

The proof that $S(Y)$ is simple is nearly identical. By \eqref{mainresorder} and \eqref{defny}, there is some $w \in \{-1,1\}$ such that
\[S(y_i) = (\overline{4i+2k},h_i+h_{3i+2k}) \text{ and } S(y_i^\prime) = (\overline{4i+2k +w},h_{i+q_0}+h_{3i+q_1+2k})\]
for every $i \in [q-k]$. Considering the parity of entries in the first coordinate, we see $S(\widehat{Y}) \cap S(Y^\prime) = \emptyset$. We then check that $S(\widehat{Y})$ and $S(Y^\prime)$ are both simple by applying Lemma \ref{techlemma} and noting that $4i+z \equiv 4j + z \modulo{t}$ if and only if $4i \equiv 4j \modulo{t}$ for every $z \in \mathbb{Z}$. 
\end{proof}

Following directly from this esult is Theorem \ref{abelianchrom}, which states that every finite Abelian group $G$ of order $n$ satisfies
\begin{equation*} \chi(G) = \begin{cases} n &\mbox{if } Syl_2(G) \text{ is either trivial or non-cyclic,}\\ n+2 &\mbox{otherwise}. \end{cases}\end{equation*}

\begin{proof}[Proof of Theorem \ref{abelianchrom}]
If $Syl_2(G)$ is trivial or non-cyclic, Theorem \ref{bigequiv} tells us that $\chi(G)=n$. Otherwise, Theorem \ref{bigequiv} tells us that $\chi(G) \geq n+2$, from which Theorem \ref{abelianmainresult} implies $\chi(G)=n+2$.
\end{proof}

\section{A general upper bound \label{generalbd}}

Consider a finite group $G$ and a normal subgroup $H \triangleleft G$. In \cite{HallPaige1955}, Hall and Paige showed that a sufficient condition for the existence of a transversal in $L(G)$ is that both $L(H)$ and $L(G/H)$ possess transversals. This turns out to be a special case of a more general result concerning colorings of finite group Cayley tables. 

Our proof of this fact relies upon a modification the mappings $R$, $C$, and $S$---introduced just before Theorem \ref{abelianmainresult}---which map sets of cells in a latin square to \emph{multi}sets of rows indices,  column indices, and symbols, respectively. Given a multiset $X$, let $\Supp(X)$ be the underlying simple set. For every set of cells $X \subseteq L(G)$, we set $R^\prime(X) := \Supp(R(X))$, $C^\prime(X) := \Supp(C(X))$, and $S^\prime(X) := \Supp(S(X))$. 

\begin{lemma}\label{quotientbnd}
Let $G$ be a finite group and let $H \triangleleft G$ be a normal subgroup. Then
\[ \chi(G) \leq \chi(H)\chi(G/H).\]
\end{lemma}

\begin{proof}

Letting $n:=|G|$ and $m:=|H|$, set $k:= \frac{n}{m}$. We begin by constructing a block representation 
\begin{equation} L(G) = \begin{pmatrix} A_{00} & A_{01} & \cdots & A_{0,k-1} \\ A_{10} & A_{11} & \cdots & A_{1,k-1} \\ \vdots & \vdots & \ddots & \vdots \\ A_{k-1,0} & A_{k-1,1} & \cdots & A_{k-1,k-1}\end{pmatrix}\label{blockrep}\end{equation}
in which each block $A_{ij}$, for $i,j \in [k]$, is a latin subsquare satisfying $\chi(A_{ij}) = \chi(H)$. Let $\{f_0,f_1,\ldots,f_{k-1}\}$ be a collection of coset representatives for $H$ in $G$, so that $G/H = \{f_0H,f_1H,\ldots,f_{k-1}H\}$. We may assume that $f_0$ is the identity element of $G$. To build the block representation \eqref{blockrep}, fix an ordering of $H = \{h_0,h_1,\ldots,h_{m-1}\}$ and order the rows and columns of $L(G)$ by 
\begin{equation}\label{roworder} h_0,h_1,\ldots,h_{m-1},f_1h_0,\ldots,f_1h_{m-1},f_2h_0,\ldots, f_2h_{m-1}, \ldots, f_{k-1}h_0,\ldots,f_{k-1}h_{m-1}.\end{equation}
Fixing arbitrary $i,j \in [k]$, we define $A_{ij}$ as the unique $m \times m$ subsquare of $L(G)$ satisfying 
\[R^\prime(A_{ij}) = \{im+x \,:\, x \in [m]\} \text{ and }C^\prime(A_{ij}) = \{jm+x \,:\, x\in [m]\}.\]
Because $H$ is normal in $G$ there is a permutation $\pi \in S_m$ such that $h_rf_j = f_jh_{\pi(r)}$ for every $r \in [m]$. Thus $S^\prime(A_{ij}) = f_if_jH$, and it follows that $A_{ij}$ is a latin subsquare of $L(G)$. 

To establish $\chi(A_{ij}) = \chi(H)$, we provide a graph isomorphism between $\Gamma(A_{ij})$ and $\Gamma(H)$. Indeed, for every $v_{ab} = (a,b,h_ah_b) \in V(\Gamma(H))$, let $\phi(v_{ab}) = (\pi^{-1}(a),b,f_if_jh_ah_b) \in V(\Gamma(A_{ij}))$. It then follows from the definition of a group that the triples $v_{ab} = (a,b,h_ah_b)$ and $v_{cd} = (r,s,h_rh_s)$ match in exactly one coordinate if and only if the corresponding triples $\phi(v_{ab})$ and $\phi(v_{rs})$ match in exactly one coordinate.

Let $K$ be the $k \times k$ array formed from \eqref{blockrep} by identifying blocks with the symbols therein contained. Having shown above that $S^\prime(A_{ij}) = f_if_jH$, we see that $K$ is a latin square which is equivalent to the Cayley table $L(G/H)$. Letting $y := \chi(G/H)$, we may select some $y$-coloring $f_\infty:K \rightarrow [y]$. Furthermore, letting $x:= \chi(H)$, we may also select an $x$-coloring $c_{ij}: A_{ij} \rightarrow [x]$. 

We now use the colorings $f_\infty$ and $\{f_{ij} \,:\, i,j \in [k]\}$ to construct $(xy)$-coloring of $L$, say $f: L \rightarrow [x] \times [y]$. For each $i,j \in [k]$ and for every cell $c \in A_{ij}\subseteq L$, set $f(c) := (f_\infty(A_{ij}),f_{ij}(c))$. See Figure \ref{dicyclic3step} for an example of a color class of $f$ when $G = \langle h,s \,|\, h^3 = s^4 = 1, s^{-1}hs=h^{-1}\rangle$ is the dicyclic group of order 12 and $H=\mathbb{Z}_3$. To see that $f$ is indeed a proper coloring, consider $c,c^\prime \in L$ such that $f(c) = f(c^\prime)$. Because $f_\infty$ is a proper coloring, $c$ and $c^\prime$ cannot lie in adjacent blocks of $V(\Gamma(K))$. They could lie in the same block, say $A_{ij}$, but because $f_{ij}$ is also a proper coloring, it is nonetheless impossible for $c$ and $c^\prime$ to be adjacent in $\Gamma(L)$.
\end{proof}

\begin{figure}
\begin{center}
\def\arraystretch{1.15}
\adjustbox{scale=1.1}{
\begin{tabular}
{|@{\hspace{2pt}}c@{\hspace{2pt}}|@{\hspace{2pt}}c@{\hspace{2pt}}|@{\hspace{2pt}}c@{\hspace{2pt}}||@{\hspace{2pt}}c@{\hspace{2pt}}|@{\hspace{2pt}}c@{\hspace{2pt}}|@{\hspace{2pt}}c@{\hspace{2pt}}||@{\hspace{2pt}}c@{\hspace{2pt}}|@{\hspace{2pt}}c@{\hspace{2pt}}|@{\hspace{2pt}}c@{\hspace{2pt}}||@{\hspace{2pt}}c@{\hspace{2pt}}|@{\hspace{2pt}}c@{\hspace{2pt}}|@{\hspace{2pt}}c@{\hspace{2pt}}|}
  \hline
$\mathbf{1}$&$h$&$h^2$&$s$&$sh$&$sh^2$&$s^2$&$s^2h$&$s^2h^2$&$s^3$&$s^3h$&$s^3h^2$\\ \hline
$h$&$\mathbf{h^2}$&$1$&$sh^2$&$s$&$sh$&$s^2h$&$s^2h^2$&$s^2$&$s^3h^2$&$s^3$&$s^3h$\\ \hline
$h^2$&$1$&$\mathbf{h}$&$sh$&$sh^2$&$s$&$s^2h^2$&$s^2$&$s^2h$&$s^3h$&$s^3h^2$&$s^3$\\ \hline\hline
$s$&$sh$&$sh^2$&$\mathbf{s^2}$&$s^2h$&$s^2h^2$&$s^3$&$s^3h$&$s^3h^2$&$1$&$h$&$h^2$\\ \hline
$sh$&$sh^2$&$s$&$s^2h^2$&$s^2$&$\mathbf{s^2h}$&$s^3h$&$s^3h^2$&$s^3$&$h^2$&$1$&$h$\\ \hline
$sh^2$&$s$&$sh$&$s^2h$&$\mathbf{s^2h^2}$&$s^2$&$s^3h^2$&$s^3$&$s^3h$&$h$&$h^2$&$1$\\ \hline\hline
$s^2$&$s^2h$&$s^2h^2$&$s^3$&$s^3h$&$s^3h^2$&$1$&$h$&$h^2$&$\mathbf{s}$&$sh$&$sh^2$\\ \hline
$s^2h$&$s^2h^2$&$s^2$&$s^3h^2$&$s^3$&$s^3h$&$h$&$h^2$&$1$&$sh^2$&$s$&$\mathbf{sh}$\\ \hline
$s^2h^2$&$s^2$&$s^2h$&$s^3h$&$s^3h^2$&$s^3$&$h^2$&$1$&$h$&$sh$&$\mathbf{sh^2}$&$s$\\ \hline\hline
$s^3$&$s^3h$&$s^3h^2$&$1$&$h$&$h^2$&$s$&$sh$&$sh^2$&$s^2$&$s^2h$&$s^2h^2$\\ \hline
$s^3h$&$s^3h^2$&$s^3$&$h^2$&$1$&$h$&$sh$&$sh^2$&$s$&$s^2h^2$&$s^2$&$s^2h$\\ \hline
$s^3h^2$&$s^3$&$s^3h$&$h$&$h^2$&$1$&$sh^2$&$s$&$sh$&$s^2h$&$s^2h^2$&$s^2$\\ \hline
\end{tabular}
}
\caption{$L(Dic_{3})$, divided into blocks as per \eqref{blockrep}, with a color class from the proof of Lemma \ref{quotientbnd} in bold.}\label{dicyclic3step}
\end{center}
\end{figure}

Recall from Theorem \ref{bigequiv} that, in determining an upper bound for the chromatic number of \emph{all} finite groups, we need only consider groups whose Sylow 2-subgroups are nontrivial and cyclic. The following structural theorem for such groups was observed in \cite{HallPaige1955} as a direct corollary of a classical result due to Burnside (\cite{HallBook} Theorem 14.3.1).
\begin{lemma}
Let $G$ be a finite group and let $P$ be a Sylow 2-subgroup of $G$. If $P$ is cyclic and nontrivial then there is a normal subgroup of odd order $H \triangleleft G$ for which $G/H \cong P$.
\label{sdp}\end{lemma}

Combining Lemmas \ref{quotientbnd} and \ref{sdp}, we obtain an upper bound for $\chi(G)$ which depends only upon the largest power of 2 dividing $|G|$. 

\begin{theorem}\label{blocklemma}
Let $l$ and $m$ be nonnegative integers such that $m$ is odd, and let $t = 2^l$. If $G$ is a group of order $n=mt$, then 
\[\chi(G) \leq \frac{t+2}{t}n= n+\frac{2n}{t}.\]
\end{theorem}

\begin{proof}
We may assume $t \geq 2$ and $Syl_2(G) = \mathbb{Z}_t$, as otherwise Theorem \ref{bigequiv} implies $\chi(G) = n \leq n + \frac{2n}{t}$. Lemma \ref{sdp} then tells us that $G$ has a normal subgroup $H$ of order $m$ satisfying $G/H \cong \mathbb{Z}_t$. With $\chi(H)$ and $\chi(\mathbb{Z}_t)$ determined by Theorem \ref{bigequiv} and Theorem \ref{abelianmainresult}, respectively, it follows from Lemma \ref{quotientbnd} that
\begin{equation*}
\chi(G) \leq \chi(H)\chi(\mathbb{Z}_t) = m(t+2) = \frac{t+2}{t}n.
\end{equation*}\end{proof}

Recall that the previously best known general upper bound was $\chi(G) \leq 2n$. Theorem  \ref{blocklemma} improves significantly upon this bound for general groups whose order is divisible by large powers of 2. Indeed, as $t$ grows with respect to $m$, Theorem \ref{blocklemma} approaches the conjectured best possible bound of $\chi(G) \leq n+2$. 

Given a (full) transversal $T$ in a latin square $L$ of order $n$, there is a unique bijection $\phi:[n]\rightarrow [n]$, known as the \defn{index map of $T$}, which sends the row index $x$ column index of the unique cell in $T$ with row index $x$, so that 
\[T = \{L_{i,\phi(i)} \,:\, i \in [n]\}.\]
If $L$ is the Cayley table of a group $G = \{g_0,g_1,\ldots,g_{n-1}\}$, then $\phi \in S_n$ is the index map of some transversal if and only if $g_ig_{\phi(i)}$ is an enumeration of $G$ (or, in the language of \cite{Evans2009}, $g_i \mapsto g_{\phi(i)}$ is a \emph{complete mapping}). 

We end by proving our second main result, Theorem \ref{brooksbound}, which states that every finite group $G$ of order $n$ satisfies
\[ \chi(G) \leq \frac{3}{2}n.\]

\begin{proof}[Proof of Theorem \ref{brooksbound}]
Let $P$ be a Sylow 2-subgroup of $G$. We may assume $P = \langle p \rangle \cong \mathbb{Z}_2$; indeed, if $P$ is trivial then Theorem \ref{bigequiv} implies $\chi(G) = n \leq \frac{3}{2}n$, while if $|P| \geq 4$ then it follows from Theorem \ref{blocklemma} that
\[ \chi(G) \leq n+\frac{2n}{|P|} \leq n+\frac{2n}{4} = \frac{3}{2}n. \]
Letting $m := \frac{n}{2}$, Lemma \ref{sdp} tells us that $G$ has a normal subgroup $H$ of order $m$ satisfying  $G/H \cong P$. As in the proof of Lemma \ref{quotientbnd}, we fix an arbitrary enumeration of $H=\{h_0,h_1,\ldots,h_{m-1}\}$ and order the rows and columns of $L=L(G)$ by 
\[ h_0,h_1,\ldots,h_{m-1},ph_0,ph_1,\ldots,ph_{m-1}.\]
Breaking $L$ into four $m \times m$ subsquares, we obtain the block representation
\[   L = \begin{pmatrix} A_0&A_2\\A_1&A_3\end{pmatrix}.\]
Observe that each $A_i$ is a latin subsquare, with $A_0 = L(H)$.

Because $m = |H|$ is odd, Theorem \ref{bigequiv} implies $\chi(A_0) = m = |A_0|$. Let $\{T_0,T_1,\ldots,T_{m-1}\}$ be an $m$-coloring of $A_0$. For each $i \in [m]$, let $\phi_i$ be the index map corresponding to the transversal $T_i$, so that
\[T_i = \{L_{j,\phi_i(j)} \,:\, j \in [m]\}.\]
We want to use $\phi_i$ to define $m$-colorings of $A_1$, $A_2$, and $A_3$. Fixing an arbitrary $i \in [m]$, define the set 
\[T^\prime_i := \{ L_{m+j,\phi_i(j)} \,:\, j \in [m]\},\]
and note that $T^\prime_i \subseteq A_1$. To see that $T^\prime_i$ is a transversal of $A_1$, note that $\phi_i$ is a bijection and $ph_jh_{\phi_i(j)}$ is an enumeration of $pH$. It is then easy to check that $\{T^\prime_i \,:\, i \in [m]\}$ is an $m$-coloring of $A_1$. 

To find $m$-colorings for $A_2$ and $A_3$, we use the fact that $H$ is normal in $G$ to define a permutation $\pi \in S_m$ for which $h_jp = ph_{\pi(j)}$ for every $j \in [m]$. Fixing an arbitrary $i \in [m]$, define the map $\psi_i : [m] \rightarrow \left([2m]\setminus [m] \right)$ by 
\[ \psi_i(j) := m + \phi_i(\pi(j)) \text{ for every }j \in [m].\]
As both $\pi$ and $\phi_i$ are permutations of $[m]$, their composition is also a permutation. Thus, $\psi_i$ is a bijection. We then define the sets
\[ Q_i := \{ L_{j,\psi_i(j)} \,: \, j \in [m]\} \text{ and } Q^\prime_i = \{L_{m+j,\psi_i(j)} \,:\, j \in [m]\}.\]
Observing that $S(L_{j,\psi_i(j)}) = h_jph_{\phi_i(\pi(j))} = ph_{\pi(j)}h_{\phi_i(\pi(j))}$
and $S(L_{m+j,\psi_i(j)}) = h_{\pi(j)}h_{\phi_i(\pi(j))}$, it is easy to check that $\{Q_i\,:\, i \in [m]\}$ and $\{Q_i^\prime \,:\, i \in [m]\}$ are $m$-colorings of $A_2$ and $A_3$, respectively.

Let $\Gamma := \Gamma(L)$. Expressing indices modulo $m$, define for every $i \in [m]$ the set
\[ X_i := T_i \cup T^\prime_{i+1} \cup Q_i \cup Q^\prime_i.\]
Because $X_i$ is the union of four subsquare transversals, it contains exactly two cells from each row, column, and symbol class of $L$. Thus, the induced subgraph $\Gamma_i := \Gamma[X_i]$ is cubic. Noticing that $\{X_i \,:\, i \in [m]\}$ partitions $L$, if we can show that $\chi(\Gamma_i) \leq 3$ for every $i \in [m]$, then we may conclude that $\chi(L) \leq \frac{3}{2}n$. Fixing an arbitrary $i \in [m]$, Brooks' Theorem tells us that $\Gamma_i$ is 3-colorable unless it contains a connected component isomorphic to the complete graph $K_4$.  

Suppose we could find a connected component $\Lambda \subseteq \Gamma_i$ which is isomorphic to $K_4$. As $V(\Gamma_i)$ is the union of four independent sets---one corresponding to each subsquare $A_j$---we may assume $V(\Lambda) = \{v_j \in A_j \,:\, j \in [4]\}$. Moreover, the row and column edges of $\Lambda$ must form a 4-cycle. Thus, if $R(v_0) = j$, we must have $C(v_0) = \phi_i(j)$ and $R(v_2) = j$. It then follows from the definition of $Q_i$ that $C(v_2) = \psi_i(j)$. But this implies $C(v_3) = \psi_i(j)$, so that $R(v_3) = m+\psi_i^{-1}\left(\psi_i(j)\right) = m+j$. We then have $R(v_1) = m+j$ and $C(v_1) = \phi_{i+1}(j)$. If $v_0,v_2,v_3,v_1$ is to form a 4-cycle, we must have $C(v_0) = C(v_1)$. However, this would mean $\phi_i(j) = \phi_{i+1}(j)$, contradicting the fact that $T_i \cap T_{i+1} = \emptyset$.
\end{proof}

\bibliographystyle{plain}
\bibliography{references}

\end{document}